\documentclass[10pt]{amsart}

\usepackage{amsfonts}
\usepackage{latexsym}
\usepackage{color}
 \RequirePackage{amsbsy} 
 \RequirePackage{amsopn} 
 \RequirePackage{amsmath}
  \RequirePackage{amssymb}
  \usepackage[mathscr]{eucal}
\definecolor{darkmagenta}{rgb}{0.5, 0, 0.5}
\definecolor{darkblue}{rgb}{0.1, 0.1, 0.7}
\definecolor{darkgreen}{rgb}{0.1, 0.35, 0.1}

 \newtheorem{thee}{Theorem}
 \newtheorem{coor}[thee]{Corollary}
 \newtheorem{leem}[thee]{Lemma}
 
 \newtheorem{prro}[thee]{Proposition}
 
 \newtheorem{exxe}[thee]{Example}
 \newtheorem{reem}[thee]{Remark}

 \newcommand{\balf}
 {\renewcommand{\theenumi}{(\alph{enumi})}
 \renewcommand{\labelenumi}{\theenumi}
                      \begin{enumerate}}
\newcommand{\ealf}   {\end{enumerate}
                      \renewcommand{\theenumi}{\arabic{enumi}}
                      \renewcommand{\labelenumi}{\theenumi.}}
\newcommand{\bara}   {\renewcommand{\theenumi}{(\arabic{enumi})}
                      \renewcommand{\labelenumi}{\theenumi}
                      \begin{enumerate} }
\newcommand{\eara}   {\end{enumerate}
                      \renewcommand{\theenumi}{\arabic{enumi}}
                      \renewcommand{\labelenumi}{\theenumi.}}

 \newcommand{\brom}   {\renewcommand{\theenumi}{(\roman{enumi})}
                      \renewcommand{\labelenumi}{\theenumi}
                      \begin{enumerate} }
\newcommand{\erom}   {\end{enumerate}
                      \renewcommand{\theenumi}{\arabic{enumi}}
                      \renewcommand{\labelenumi}{\theenumi.}}
					  
	 \newcommand{\rc}{\color{red}}

	 \newcommand{\ec}{\color{black}}

	  \newcommand{\Na}{\mbox{\rm Na}}

	   \newcommand{\Max}{\mbox{\rm Max}}
     \newcommand{\f}{\boldsymbol{f}}   
       \newcommand{\F}{\boldsymbol{F}}   
	     \newcommand{\starf} {\star_{_{\! f}}} 
	      \newcommand{\start} {\widetilde{\star}} 
	       \newcommand{\stara} {\star_{_{\! a}}} 
	        \newcommand{\astf} {\ast_{_{\! f}}} 
	         \newcommand{\astt} {\widetilde{\ast}} 
	           \newcommand{\QMax}{\mbox{\rm QMax}}   
  \newcommand{\co} {\boldsymbol{c}} 

  \DeclareMathOperator{\calS}     {\mathcal S}%
  \DeclareMathOperator{\calbV} {\boldsymbol{\mathcal V}}%
  \DeclareMathOperator{\calbW}     {\boldsymbol{\mathcal W}}%
  \DeclareMathOperator{\calJ} {\mathcal J}%
  \RequirePackage{amssymb} 
  \RequirePackage{amsmath} 

\begin{document}

 \title[]{Cancellation properties in ideal systems:\\ A classification of $\boldsymbol{e.a.b.}$ semistar operations}

 \author{Marco Fontana and K. Alan Loper}
 
 \dedicatory{\sl to Paulo Ribenboim on the occasion of his 80th birthday}

\thanks{\it  Acknowledgments. \rm 
During the preparation of this paper, the first named author
was partially supported by  a research grant PRIN-MiUR}

\address{M.F.: \ Dipartimento di Matematica, Universit\`a degli Studi
``Roma Tre'', 00146 Rome, Italy.}
\email{fontana@mat.uniroma3.it }
\address{A.L.: \ Department of Mathematics, Ohio State University, Newark, Ohio 43055, USA.}
\email{lopera@math.ohio-state.edu}

\date{\today} 
 

 \maketitle

 \begin{abstract} We give a classification of {\texttt{e.a.b.}} semistar  (and star)   operations by defining four different (successively smaller) distinguished classes. Then, using a standard notion of equivalence of semistar (and star) operations to partition the collection of all {\texttt{e.a.b.}}  semistar  (or star) operations,  we show that there is exactly one operation of finite type in each equivalence class and that this operation has a range of nice properties.  We give examples to demonstrate that the four classes of {\texttt{e.a.b.}} semistar   (or star)  operations   we defined can all be distinct. In particular, we solve the open problem of showing that  {\texttt{a.b.}}   is really a stronger condition than {\texttt{e.a.b.}}   \end{abstract}

 
 \section{Introduction}
 
 
 In the classical (Krull's) setting, the study of Kronecker function rings on an integral domain generally focusses on the collection of  {\texttt{a.b.}} ({\sl = arithmetisch brauchbar})  star operations on the domain.   Gilmer's presentation of star operations 
 \cite[Section 32]{G} covers  the class of {\texttt{a.b.}}   star operations and also the (presumably  larger  class of)  {\texttt{e.a.b.}}   ({\sl = endlich arithmetisch brauchbar})  star operations   (the definitions are recalled in the following section).  
 
 This paper began with an attempt to   clarify the relation between 
 the {\texttt{e.a.b.}} and {\texttt{a.b.}} conditions and trying to solve the open problem of showing that  {\texttt{a.b.}}   is really a stronger condition than {\texttt{e.a.b.}}    In Section 2 of the paper we give some general background and prove some elementary results concerning star operations (and the more general concept of semistar operations) and the related issue of cancellation properties of ideals  (since the {\texttt{e.a.b.}}   condition is essentially a cancellation property).  We also 
expand our goal and define four different (successively smaller) classes of {\texttt{e.a.b.}} semistar  (and star)   operations.  Given two {\texttt{e.a.b.}} semistar operations, we say that they are equivalent if they agree on the class of all finitely generated ideals.  Using this notion of equivalence to partition the collection of all {\texttt{e.a.b.}}  semistar  (or star) operations,   we show that there is exactly one operation of finite type in each equivalence class and that this operation has a range of nice properties.  These operations of finite type constitute the smallest of our four classes.  Then, in Section 3, we give examples to demonstrate that the four classes of semistar  (or star)  operations  we defined can all be distinct, including the motivating example of a star operation that is {\texttt{e.a.b.}}  but not  {\texttt{a.b.}}  Then, in a brief final section, we approach the question of generalizing the results beyond the scope of {\texttt{e.a.b.}}  operations.  In particular, we note that for general star or semistar operations, an operation of finite type may not have the various nice properties that an  {\texttt{e.a.b.}}  operation of finite type has.  We suggest an alternative construction to the standard finite-type construction which agrees with the finite-type construction in the {\texttt{e.a.b.}}  case and does appear to give results similar to our {\texttt{e.a.b.}}  results in the general setting.  This generalization is based on recent results from \cite{FL:2007}.

\section{Classification of  ${e.a.b.}$ semistar operations and Cancellation properties}


Let $D$ be an integral domain with quotient field $K$. 
Let  $\boldsymbol{{F}}(D) $   [respectively,  $\boldsymbol{{f}}(D)$] be 
the set of all nonzero  fractionary ideals [respectively, nonzero finitely 
generated fractionary ideals] of $D$.  Let   $\,\boldsymbol{\overline{F}}(D)\,$  
represent the set of all nonzero $\,D$--submodules of $\,K\,$ (thus,
${\boldsymbol{f}}(D) \subseteq {\boldsymbol{F}}(D) \subseteq \boldsymbol{\overline{F}}(D)$).

 W. Krull  introduced the concept of a star operation in 1936 in his first Beitr\"age paper \cite{Krull:1936} (or \cite{Krull}).  In  1994, {Okabe} and 
{Matsuda} introduced the more
``flexible'' notion of semistar operation $\,\star\,$ of an integral
domain $\,D\,,$ \ as a natural generalization of the notion of star
operation, allowing $\, D \neq D^\star$.

A mapping   $\,\star : 
\boldsymbol{\overline{F}}(D) \rightarrow
\boldsymbol{\overline{F}}(D)\,$,  $\,E \mapsto E^\star,$
is called {\it a semistar operation of $\,D\,$} \rm 
if, for
all $\,z \in K\,$, $\,z \not = 0\,$ and for all $\,E,F \in
\boldsymbol{\overline{F}}(D)\,$, the following properties hold: $\mathbf{ \bf (\star_1)} \;(zE)^\star = 
zE^\star \,; $ \,
  $\mathbf{ \bf (\star_2)} \;   E \subseteq F 
\;\Rightarrow\; E^\star
\subseteq
F^\star \,;$ \,
$\mathbf{ \bf (\star_3)}\;    E \subseteq 
E^\star 
 \textrm {  and  } E^{\star \star} := (E^\star)^\star = E^\star\,. $

\smallskip

When $D^\star=D$, the map $\star$, restricted to $\boldsymbol{F}(D)$, defines  
a star
operation\!\! 
\footnote{More explicitly,  a \emph{star operation $\ast$ of an integral domain $D$} is a mapping $\,\ast : 
\boldsymbol{{F}}(D) \rightarrow
\boldsymbol{{F}}(D)\,$,  $\,E \mapsto E^\ast\,$ such that the following properties hold: $\mathbf{ \bf (\ast_1)}$ $(zD)^\ast = z D$ and $(zE)^\ast = zE^\ast$,\  $\mathbf{ \bf (\ast_2)}$    $E \subseteq F \Rightarrow E^\ast
\subseteq
F^\ast$, \ 
 $\mathbf{ \bf (\ast_3)}$  $E \subseteq 
E^\ast 
 \textrm {   and  
}  E^{\ast \ast} := (E^\ast)^\ast = E^\ast\,,$
 for
all nonzero $\,z \in K\,$, and for all $\,E,F \in
\boldsymbol{{F}}(D)\,$. }
of $D$ \cite[Section 32]{G};   in this situation, we say that $\star$ is a \emph{(semi)star operation of} $D$.  A \emph{proper semistar operation of $D$} is a semistar operation $\star$ of $D$ such that $D \subsetneq D^\star$.   

 \rm  For several examples we construct, we use results that were proven for star operations rather than semistar.   However, if $\ast$ is a star operation on an integral domain $D$ (hence, defined only on $\boldsymbol{F}(D)$),  we can extend it trivially to a semistar   (in fact, (semi)star) operation of $D$, denoted $\ast_e$, by defining $E^{\ast_e}$ to be the quotient field of $D$ whenever 
$E \in \boldsymbol{\overline{F}}(D) \setminus \boldsymbol{F}(D)$.  Hence, our star operation examples can be considered to be semistar examples as well.  

\smallskip

As in the classical star-operation setting, we associate to a \sl 
semistar \rm  ope\-ra\-tion $\,\star\,$ of $D$ a new semistar 
operation
{$\,\star_{_{\! f}}\,$} of $D$ as follows.   If $\,E \in
\boldsymbol{\overline{F}}(D)\,$ we set:
$$
E^{\star_{_{\! f}}} := \bigcup \{F^\star \;|\;\, F \subseteq E,\, F \in
\boldsymbol{f}(D)
\}\,.
$$
We call {$\, {\star_{_{\! f}} }\,$ \it the 
semistar
operation of finite type of $D$ \rm associated to $\,\star\,$}.
If $\,\star =
\star_{_{\! f}}\,$,\ we say that $\,\star\,$ is {\it a 
semistar ope\-ra\-tion of finite type on $D$.  \rm {\rc Given two semistar operations $\star'$ and $\star''$ of $D$, we say that $\star' \leq \star''$ if $E^{\star'} \subseteq E^{\star''}$ for all $E \in
\boldsymbol{\overline{F}}(D)$}. Note that ${\star_{_{\! f}} } \leq \star$
 and
$\,(\star_{_{\! f}})_{_{\! f}} = \star_{_{\! f}}\,$, \ so $\,\star_{_{\! f}}\,$ is a semistar operation
of finite type of $\,D\,.$ 

 If $\star$ coincides with the  semistar \emph{$v$--operation of $D$}, defined by $E^{v} : = (D:(D:E))$, for each 
   $E \in
\boldsymbol{\overline{F}}(D)$, then  $v_f$ is denoted by $t$. Note that $v$ [respectively, $t$] restricted to $\boldsymbol{{F}}(D)$ coincides with the classical {\sl star} $v$--operation [respectively, $t$--operation]  of $D$.

Let 
     $\star$ be a semistar operation on $D$. If $F$ is in $\boldsymbol{f}(D)$, we say that 
     $F$ is 
            \it  $\star$--\texttt{eab} \rm [respectively,  \it  
             $\star$--\texttt{ab}\rm]
  if
$(FG)^{\star}
             \subseteq (FH)^{\star}$ implies that $G^{\star}\subseteq H^{\star} $, with $G,\ H \in 
\boldsymbol{f}(D)$, [respectively,  with $G,\ H \in 
\overline {\boldsymbol{F}}(D)$].
       
      The operation $\star$ is said to be   \it \texttt{eab} \rm [respectively, \it \texttt{ab}\rm\ \!]  if each $F\in \boldsymbol{f}(D)$ is $\star$--\texttt{eab}  [respectively, $\star$--\texttt{ab}]. \  An \texttt{ab} operation is obviously an \texttt{eab} operation. 
      
        \begin{reem} \rm    W. Krull,  in \cite{Krull:1936}, only considered the concept of 
 ``{\bf\texttt{a}}rithmetisch {\bf\texttt{b}}rauchbar'' (for short, \texttt{a.b.} or, simply \texttt{ab} as above) $\star$--operation  (more precisely,  Krull's original notation was  ``\ $^{\prime}$--Operation'', instead of  ``$\star$--operation''). He did not consider the concept of    ``{\bf\texttt{e}}ndlich {\bf\texttt{a}}rithmetisch {\bf\texttt{b}}rauchbar''   $\star$--operation.  
 
The \texttt{e.a.b.} (or, more simply, \texttt{eab} as above)   concept stems from the original version of
Gilmer's book \cite{G68}. The results of Section 26 in \cite{G68}  show that this 
(presumably) weaker
concept is all that one needs to develop a complete theory of Kronecker
function rings. 
       Robert Gilmer  explained to us that  \   \ $\ll$ I believe I was influenced to recognize this because
during the 1966 calendar year in our graduate algebra seminar (Bill
Heinzer, Jimmy Arnold, and Jim Brewer, among others, were in that
seminar) we had covered Bourbaki's Chapitres 5 and 7 of 
\it Alg\`ebre Commutative\rm , and the development in Chapter 7 on the $v$--operation indicated
that\texttt{ e.a.b.} would be sufficient.$\gg$
\end{reem}

\begin{reem} \label{rk:2}
\bf (1) \rm When $\star$ coincides with the identity star operation $d$  on the integral domain $D$,  the notion of   $d$--\texttt{eab} 
[respectively,  $d $--\texttt{ab}], 
\sl for finitely generated ideals, \rm  coincides with the notion of 
\emph{quasi--cancellation ideal} [respectively, \emph{cancellation ideal}] studied by  D.D. Anderson and D.F. Anderson 
\cite{AA} (cf. also \cite{G65}).

As a matter of fact, a nonzero ideal $I$  (not  necessarily 
finitely generated)  of an integral domain 
$D$ is called a \it cancellation  \rm[respectively, \it quasi--cancellation\rm] 
\it ideal of $D$ \rm if $(IJ:I) = J$, for each nonzero  ideal $J$ 
of $D$ [respectively, if $(IF:I) = F$, for each nonzero finitely generated ideal $F$ 
of $D$].\\
 Obviously, a cancellation ideal is a quasi--cancellation 
ideal, but in general (for non finitely generated ideals)  the converse does not hold (e.g., a maximal 
ideal of a nondiscrete rank one valuation domain, \cite{AA}).

   For a  finitely generated ideal, the notion of  cancellation 
ideal coincides with the notion of quasi--cancellation ideal 
\cite[Corollary 1]{AA} (thus, in particular, the identity operation $d$ is \texttt{eab} if and only if $d$ is \texttt{ab} and this happens if and only if $D$ is a Pr\"ufer domain (cf. also the following part \bf (4)\rm). More precisely, by \cite[Lemma 1 and Theorem 1]{AA}  we have:

\noindent \it If $I$
is a nonzero \rm finitely generated 
ideal \it of an integral domain $D$, then the following conditions are equivalent:
\begin{enumerate}
\item[]
{
\begin{enumerate}
    \bf \item [(i)] \it  $I$ is a quasi--cancellation ideal of $D$;
    
    \bf \item [(ii)] \it $IG
             \subseteq IH$, with $G$ and $H$ nonzero 
             finitely generated ideals of $D$, implies
             that $G\subseteq H$;
             
             \bf \item [(iii)] \it $IG
             \subseteq IH$, with $G$ and $H$ nonzero 
              ideals of $D$, implies
             that $G\subseteq H$;

             \bf \item [(iv)] \it $I$ is a cancellation ideal of $D$;
             
            \bf \item [(v)] \it for each prime [maximal] ideal $Q$ of $D$, 
             $ID_{Q}$ is an invertible ideal of $D_{Q}$;

    \bf \item [(vi)] \it $I$ is an invertible ideal of $D$.
    
       \end{enumerate} 
       }
       \end{enumerate}\rm 
       
     \ec~Note that the definitions of quasi-cancellation and cancellation ideal   can be  extended  in a natural way to the case of fractional ideals and, {\sl mutatis mutandis}, the previous equivalent conditions hold for  fractional ideals.

 \bf (2) \rm The notion of   quasi-cancellation ideal   was introduced in \cite{AA}, in relation to the fact that in \cite[Exercise 4, page 66]{G} it was erroneously stated that a nonzero ideal $I$ of  an integral domain $D$ is a cancellation ideal if and only if $(IF:I)=F$, for each finitely generated ideal $F$ of $D$ (see the  counter-example mentioned in part \bf (1)\rm).

\bf (3) \rm  Kaplansky,  in an unpublished set of notes \cite[Exercise 7, page 67]{G}, proved a result that, in the integral domain case, affirms that   \it a nonzero finitely generated ideal $I$ of a local integral domain $D$ is a cancellation ideal if and only if  $I$ is principal.  \rm Therefore, the equivalence (\bf (iv)$\Leftrightarrow$(v)\rm) in part \bf (1) \rm is a ``globalization'' of Kaplansky's result.
Note also that Kaplansky observed that, if $I, G$ and $H$ are nonzero  ideals of an integral domain $D$ with $IG\subseteq IH$ and if $I$ is finitely generated ideal, generated by $n$ elements, then necessarily $G^n \subseteq H$ \cite [Theorem 254]{Ka:1971}.

\bf (4) \rm  Recall that  Jaffard   \cite{J:1960} proved that  
 \it  for each   ideal $I \in \boldsymbol{f} (D)$, $I$ is a (quasi--)cancellation ideal  if and only if $D$ is a Pr\"ufer domain \rm   
(cf. also  Jensen   \cite[Theorem 5]{Jensen:1963};  in that paper   Jensen   \cite[Theorem 6]{Jensen:1963} proved also that   
\it  for each  ideal $I \in \boldsymbol{F} (D)$, $I$ is a cancellation ideal    if and only if $D$ is an almost Dedekind domain). \rm
Recall also that, by \cite[Theorem 7]{AA}, 
 \it $I$ is a quasi--cancellation ideal, for each $I \in \boldsymbol{F} (D)$, if and only if $D$ is a completely integrally closed Pr\"ufer domain.  \rm

\bf (5) \rm Note that, \it  when $D$ is  a Pr\"ufer domain\rm , it is known \cite[Theorem 2 and Theorem 5]{AA} that:
\begin{enumerate}

\item[\bf (5, a)\rm] \it $I \in \boldsymbol{F} (D)$ is a quasi--cancellation ideal $\Leftrightarrow$ $(I:I) =D$. \rm 
 
\item[\bf (5, b)\rm] \it $I \in \boldsymbol{F} (D)$ is a cancellation ideal $\Leftrightarrow$ $ID_M$ is principal for each $M$ maximal ideal of $D$. \rm
\end{enumerate} 

 \noindent D.D. Anderson and Roitman  \cite[Theorem]{AR:1997} extended \bf (5, b) \rm outside of the (Pr\"ufer) domain case and  proved that,  \it  given a nonzero ideal [respectively, a regular ideal]  $I$ of an integral domain [respectively,  a ring] $R$, then 
 $I$  is a cancellation ideal of $R$ if and only if $IR_M$ is a   principal [respectively, principal regular]   ideal of $R_M$,    for each  maximal ideal $M$ of $R$. \rm   
 
Note that the previous statement was ``extended''  further    to   submodules of the quotient field of an integral domain $D$ by  Goeters and Olberding  \cite{GO:2000}.
 Let $E \in  \overline{\boldsymbol{F}}(D)$, $E$ is called  \it a cancellation module for $D$  \rm  if, for $G, H \in \overline{\boldsymbol{F}}(D)$, 
 $EG = EH$ implies that $G =  H$. 
Then, by \cite[Theorem 2.3]{GO:2000},  \it $E$ is a cancellation module for $D$ if and only if   $ED_M$ is principal,   for each $M \in \mbox{\rm Max}(D)$, or, equivalently, if and only if $ED_M$ is a cancellation  module for $D_M$,   for each $M \in \mbox{\rm Max}(D)$.  \rm   
  
\end{reem}

\medskip

  We  note  that if $\star$ is an \texttt{eab} semistar operation then $\, {\star_{_{\! f}} }\,$ is also an \texttt{eab} semistar operation, since they agree on all finitely generated ideals.  
The following easy result generalizes the fact, already observed in Remark \ref{rk:2}(1),  that the identity semistar operation   $d$ is \texttt{eab}  if and only if   it    is \texttt{ab}. 

%

  \begin{leem} \label{le:2}   Let $\star$ be a semistar operation of finite type, then
$ \star$  is an \texttt{eab} semistar operation if and only if  $\star$ is an \texttt{ab}  semistar operation.
\rm \end{leem} 

\begin{proof} Since  it is obvious that an \texttt{ab} semistar operation is always \texttt{eab}, we need only to prove the converse.  Let $I  \in \boldsymbol{f}(D)$ and $J, L \in\boldsymbol{\overline{F}}(D)$.  Assume that $(IJ)^\star \subseteq (IL)^\star$. By the assumption, we have
$(IJ)^\star =   \bigcup\{H^\star \mid H\in  \boldsymbol{f}(D) \,,\; H\subseteq IJ\}= \bigcup\{(IF)^\star \mid F\in  \boldsymbol{f}(D)\,,\; F\subseteq J\}$ 
 and similarly   $(IL)^\star = \bigcup\{(IG)^\star \mid G\in  \boldsymbol{f}(D)\,,\; G\subseteq L\}$. Therefore, for each  $F\in  \boldsymbol{f}(D)\,,\; F\subseteq J$, we have $IF \subseteq \bigcup\{(IG)^\star \mid G\in  \boldsymbol{f}(D)\,,\; G\subseteq L\}$. Thus we can find $G_1, G_2, ..., G_r$ in $\boldsymbol{f}(D)$ with the property that $G_i \subseteq L$,  for $1 \leq i\leq r$, such that:
$$
(IF)^\star \subseteq (IG_1 \cup IG_2 \cup ...\cup IG_r)^\star \subseteq (I(G_1 \cup G_2 \cup ...\cup G_r))^\star\,.$$
Since $\star$ is an \texttt{eab} semistar operation then $F^\star \subseteq (G_1 \cup G_2 \cup ...\cup G_r)^\star \subseteq \bigcup \{ G^\star  \mid G \in  \boldsymbol{f}(D)\,,\; G \subseteq L\} = L^{\star_{_{\!f}}} =L^\star$ and so
$J^\star =J^{\star_{_{\!f}}} = \bigcup \{ F^\star  \mid F\in  \boldsymbol{f}(D)\,,\; F\subseteq J\} \subseteq L^\star$.
\end{proof}

The next result provides a useful generalization of   Lemma \ref{le:2}.

	      \begin{prro} \label{pr:4}  Let $D$ be an integral domain, let 
         $\star$ be a semistar operation on $D$, and let  $F \in \boldsymbol{f}(D)$. Then
$F$ is $\star$--\texttt{eab} if and only if  F is  $\star_{_{\! f}}$--\texttt{(e)ab}.     In particular, the notions of $\star$--\texttt{eab} semistar operation and   $\star_{_{\! f}}$--\texttt{(e)ab}  semistar operation coincide.
\end{prro}

\begin{proof} Since from the definition it follows that the notion of $\star$--\texttt{eab} coincides with the notion of $\star_{_{\! f}}$--\texttt{eab}  and, by  Lemma \ref{le:2},   the notion of $\starf$--\texttt{eab} coincides with the notion of $\star_{_{\! f}}$--\texttt{ab},   it remains to show that if $F$ is $\star$--\texttt{eab} then   $F$ is $\star_{_{\! f}}$--\texttt{ab}, when $F$ belongs to  $\boldsymbol{f}(D)$. Let $G, H \in 
\overline{\boldsymbol{F}}(D)$ and assume that 
$(FG)^{\star_{_{\! f}}} \subseteq (FH)^{\star_{_{\! f}}}$, then arguing as in Lemma  \ref{le:2}, 
for each  $G' \in \boldsymbol{f}(D)$, with $G'\subseteq G$, we can find a $H'_{G'} \in \boldsymbol{f}(D)$, with 
$H'_{G'} \subseteq H$, in such a way that $(FG' )^\star \subseteq (FH'_{G'})^\star$. Since $F$ is $\star$--\texttt{eab}, then ${(G' )}^\star \subseteq (H'_{G'})^\star$ and so $G^{\star_{_{\! f}}} = \bigcup \{ {(G' )}^\star \mid 
G' \in \boldsymbol{f}(D)\,,\; G'\subseteq G\} \subseteq
\bigcup \{ (H'_{G'})^\star \mid 
G' \in \boldsymbol{f}(D)\,,\; G'\subseteq G\} \subseteq 
\bigcup \{ {(H' )}^\star \mid 
H' \in \boldsymbol{f}(D)\,,\; H'\subseteq H\}= H^{\star_{_{\! f}}}$. 
\end{proof}

 \bigskip

 If   $\calbW$ is a given family of valuation {\rc overrings} of $D$, 
     then  the mapping $\wedge_{\calbW}$   defined as follows: for each $E\in \boldsymbol{\overline{F}}(D)$,
     $$
E^{\wedge_{\calbW}} := \bigcap\{EW \mid W \in \calbW \}
$$
          defines  an \texttt{ab} semistar operation of $D$, since $FW$ is principal in $W$, for each $F\in \boldsymbol{f}(D)$ and for each $W \in \calbW$.  We call a semistar operation of the previous type a \emph{${\calbW}$\!--operation of $D$}.  If ${\calbW}$ coincides with the set ${\calbV}$ of all valuation overrings of $D$, then  we call ${\wedge_{\calbV}}$  \emph{ the $b$--operation of $D$}.

  If we assume that, given  a family of valuation overrings overrings $\calbW$ of $D$,  the overring $T:=\bigcap\{W \mid W \in \calbW\}$ of $D$  coincides with $D$, then  the map   $\wedge_{\calbW}$  defines a (semi)star operation of $D$.  In particular, if (and only if) $D$ is integrally closed, the $b$--operation is a  (semi)star operation of $D$.    \ec

\begin{reem} \rm             
Given an integrally closed domain $D$, note that Gilmer discusses star operations defined as  above (on the fractional ideals of $D$) using  collections of valuation overrings $\calbW$ of  $D$ with the property that  $D =\bigcap\{W \mid W \in \calbW\}$  and refers to them as 
\emph{$w$--operations}.    Since the terminology of $w$--operation was re-introduced recently by Wang and McCasland (see \cite{WMc1} and  \cite{WMc2}) for denoting a very different kind of star operation,    in order to avoid a possible confusion, we have slightly modified Gilmer's original terminology (i.e., star ``$\calbW$--operation'' instead of ``$w$--operation''), by emphasizing the set of valuation overrings   occurring   in the definition.
\end{reem}    

Gilmer proves   that, given   any \texttt{eab} star operation $\ast$ of a domain $D$, there exists a (star) ${\calbW}$\!--operation  of $D$ which agrees with $\ast$ on all finitely generated ideals \cite[Theorem 32.12]{G}.  
It would seem then that  ${\calbW}$\!--operations may be the most refined class of \texttt{eab}  operations.   We have one more class to define, however.

For a domain
$\,D\,$ and a semistar operation $\,\star\,$ of $\,D\,$, \ we say 
that 
a
valuation overring $\,V\,$ of $\,D\,$ is \it a
$\star$--valuation overring of $\,D\,$ \rm provided $\,F^\star
\subseteq FV\,,$ \ for each $\,F \in \boldsymbol{f}(D)\,$.  
Set $\calbV(\star) := \{ V \mid V  \mbox{ is a $\star$--valuation  overring of $D$}\}$ and let ${b(\star)}:= \wedge_{\calbV(\star)}$  the \texttt{ab} semistar operation on $D$ defined as follows: for each $E\in \overline{\boldsymbol{F}}(D)$,
$$
E^{{b(\star)}}:= \bigcap\{EV \mid V \in \calbV(\star)\}\,.
$$
 Clearly, when $\star$ coincides with $d$, the  identity  (semi)star operation, then $b(d) =b$.    Note that, {\rc this} example shows that even if $\star$ is a (semi)star operation, $b(\star)$ may be a {\sl proper} semistar operation {\rc (e.g.,  $b(d) =b$ is a (semi)star operation of $D$ if and only if $D$ is integrally closed)}.

We call the semistar operation $b(\star)$ defined as above, using the $\star$--valuation overrings of a domain $D$ associated with a given semistar operation $\star$ on $D$, \emph{the complete $\calbW$--operation associated with $\star$}.    From the definition, it follows that $\starf \leq b(\star)$.  \emph{A complete \texttt{ab} operation}  is a semistar operation $\star$ such that $\star = b(\star)$.    Clearly, a complete \texttt{ab} operation is a $\calbW$--operation and so, without loss of generality, we may consider just the complete $\calbW$--operations.   Since $F^{b(\star)}V = F^{\star}V$, for all $F\in \f(D)$ and for all $V \in \calbV(\star)$, then clearly, $b(b(\star)) = b(\star)$ and so $b(\star)$ is a complete \texttt{ab} operation.

Let $D$ be a domain and $\star$ a semistar operation.  Note that,
by definition, the $\,\star$--valuation overrings coincide with the
$\,\star_{_{\! f}}$--valuation overrings.   Hence, the above construction could be done using $ {\star_{_{\! f}} }$ in place of $\star$, i.e., $b(\star) = b(\star_{_{\! f}})$.

\begin{reem}  \rm Note that not all $\calbW$--operations are complete (see  Example \ref{notcompleteW}).  In a work in progress on the ultrafilter topology of abstract Riemann surfaces (in the sense of Zariski  {\cite{ZS}}), we will describe the complete \texttt{ab} semistar operation $b({\wedge_{\calbW}})$ for any family $\calbW$ of valuation domains sharing the same field of quotients.

\end{reem}

\medskip 

 The four distinguished classes of  semistar operations introduced above are related as follows.

 \begin{prro}  \label{eab} Let $D$ be a domain and let $\star$ be a semistar operation on $D$.  Consider the following four propositions.
 
 \begin{enumerate}
 
 \item $\star$ is an \texttt{eab} operation.
 
 \item $\star$ is an \texttt{ab} operation.
 
 \item $\star$ is a $\calbW$--operation.
 
 \item $\star$ is a complete $\calbW$--operation.
 
 \end{enumerate}
 
 Then $(4) \Rightarrow (3) \Rightarrow (2) \Rightarrow (1)$.
 \end{prro}
 
 \begin{proof}  The only implication which is not trivial is $(3) \Rightarrow (2)$.  This follows immediately though from the observation that any finitely generated ideal of a domain $D$ extends to a principal ideal in any valuation overring $V$ of $D$.
 \end{proof}

  The next goal is to give examples to show that each of the implications is not reversible.    In fact, this paper began with the desire to demonstrate that the \texttt{ab} property was properly stronger than the \texttt{eab} property (we were unable to find an example in the literature of an \texttt{eab} operation which was not \texttt{ab}) and expanded to a broader   study and finer classification of \texttt{eab} operations.  
In particular,   we pay special   attention to the class of complete $\calbW$--operations and give several characte\-ri\-zations of them.

 It is not so simple to demonstrate that the implications in Proposition \ref{eab}  are not reversible.  We will give three examples covering the three pairs of classes, including the desired example of an \texttt{eab} operation which is not an \texttt{ab} operation.  First, however, we will give the promised additional characterizations of the class of complete $\calbW$--operations.

 \medskip

  We start by extending   Gilmer's    notion of equivalent  star operation to the semistar setting: let $\star_1$ and $\star_2$ be two semistar operations defined on an integral domain $D$,  we say  that \emph{$\star_1$ is equivalent to $\star_2$} if they agree on $\f(D)$, i.e., $F^{\star_1} = F^{\star_2}$ for each $F \in \f(D)$.   It is very plausible that there can be numerous  \texttt{eab}  semistar operations that are all equivalent to the same \texttt{(e)ab}  semistar operation of finite type.   Note that, if $\star_1$ and $\star_2$ are equivalent and $\star_1$ is \texttt{eab},  then $\star_2$ is also \texttt{eab}.      Hence, we can partition the set of all \texttt{eab}  semistar operations on a domain $D$ into classes of   equivalent operations.   Each  equivalence class  has a single distinguished member, the one of finite type.

A result  proven in \cite[Proposition 3.4 and Theorem 3.5]{FL1} ensures that each \texttt{eab} semistar  operation $\star$ is equivalent to $b(\star)$. 
The preceding fact seems to give us two distinguished members (i.e., $\starf$ and $b(\star)$) in each equivalence class of \texttt{eab} semistar  operations.  We resolve this apparent conflict by introducing yet another semistar construction.

 Suppose that $D$ is a domain with quotient field $K$, $\star$ is a semistar operation on $D$, $X$  is an indeterminate over $D$ and $\boldsymbol{c}(h)$ is the content of a polynomial $h \in D[X]$.  Then,  we define  
 
 $$
\begin {array} {rl}
\mbox{Kr}(D,\star) := \{ f/g  \, \ |& \, f,g \in D[X], \ g 
\neq 0, 
\;
\mbox{ and there exists }\\ & \hskip -4pt h \in D[X] \setminus \{0\} 
\; \mbox{ 
with } (\boldsymbol{c}(f)\boldsymbol{c}(h))^\star
\subseteq (\boldsymbol{c}(g)\boldsymbol{c}(h))^\star \,\}.
\end{array}
$$
 
This is a B\'ezout domain with quotient field $K(X)$, called \it  the semistar Kronecker function ring  associated to  semistar operation $\star$ \cite[Theorem 5.1 and Theorem 3.11 (3)]{FL1}.  \rm  We can then define
 an \texttt{eab} semistar operation on $D$, denoted by ${\circlearrowleft_{\mbox{\tiny Kr}}}$,   as follows: 
 $$
 E^{\circlearrowleft_{\mbox{\tiny Kr}}}:= E\mbox{Kr}(D, \star) \cap K\,, \quad \ \mbox{for each $E \in \overline{\boldsymbol{F}}(D)$\,, \cite[Corollary 5.2]{FL1}}\,.$$

 From \cite[Proposition 3.4 and its proof]{FL:2001b}   (or, in a more general context, from \cite[Proposition 6.3]{FL:2007})   it follows that, given an \texttt{eab} semistar operation $\star$, ${\circlearrowleft_{\mbox{\tiny\rm Kr}}} =b(\star)$.

On the other hand, it is proven in   \cite[Corollary 5.2]{FL1}   that ${\circlearrowleft_{\mbox{\tiny Kr}}}$ is a semistar operation of finite type.     Hence, the preceding    results, all collected, yield the following: for any  \texttt{eab} semistar operation $\star$, ${\circlearrowleft_{\mbox{\tiny\rm Kr}}} = b(\star)
= \star_{_{\! f}}$.

\smallskip

There is still another  construction, with a more classical origin, for associating to a semistar operation an \texttt{(e)ab} semistar operation of finite type. In order to introduce this construction we need first to generalize, in the semistar operation setting, one of the useful characterizations,  given in \cite[Theorem 6.5]{G} and \cite[Lemma 1]{AA} for cancellation and  quasi--cancellation ideals.  

\begin{leem} \label {le:3} \sl Let $D$ be a domain, let $F \in \boldsymbol{f}(D) 
$ and let $\star$ be a semistar operation on $D$. Then, 
   \sl $F$ is $\star$--\texttt{eab}  [respectively, $\star$--\texttt{ab}]  if and only if  $\left(
               (FH)^\star : F^\star \right) = H^\star$, for each $H \in
               \boldsymbol{f}(D)$ [respectively, for each $H \in
\overline{ \boldsymbol{F}}(D)$]. \rm \end{leem} 

 \noindent \ec~(Note that  $\left((FH)^\star : F^\star \right) = \left(
               (FH)^\star : F\right)$, so the previous equivalences can be stated in a formally slightly different way.)
               
                \begin{proof}  We consider only  the \texttt{ab} case, since $\star$--\texttt{eab} coincides with $\star_{_{\! f}}$--\texttt{ab}  (Proposition \ref{pr:4}). As a matter of fact, 
               $\left(
               (FH)^\star : F^\star \right) = H^\star$,  coincides with $\left(
               (FH)^{\star_{_{\! f}}}: F^{\star_{_{\! f}}}\right) = H^{\star_{_{\! f}}}$, when  $F, H \in
               \boldsymbol{f}(D)$; \ if  $H \in
\overline{ \boldsymbol{F}}(D)$ and    $F\in
               \boldsymbol{f}(D)$, then $\left(
               (FH)^{\star_{_{\! f}}}  : F^{\star_{_{\! f}}}  \right)  = \left(
               (FH)^{\star_{_{\! f}}}  : F \right) = \left((\bigcup \{
               (FL)^{\star} \mid L\subseteq H\,,\; L \in \boldsymbol{f}(D)\}) : F \right) = 
              \bigcup \{  \left(
               (FL)^{\star}: F \right) \mid L\subseteq H\,,\; L \in \boldsymbol{f}(D)\}$.  
               
               The ``if'' part:  it is easy to see that, $F$ is 
$\star$--\texttt{ab} if and
             only if $(FG)^{\star} =(FH)^{\star}$, with $G,\ H \in
                         \overline{ \boldsymbol{F}}(D)$, implies that $G^{\star}= H^{\star}$.  
Then,
             $$
             (FG)^{\star}
             =(FH)^{\star} \; \Rightarrow \; \left((FG)^{\star}: 
F^\star\right)=
             \left((FH)^{\star}: F^\star\right)\,.$$
             The conclusion now is a straightforward consequence of 
             the assumption.
             
    The ``only if'' part:  given   
             $H  \in\overline{ \boldsymbol{F}}(D)$, clearly  
	     $H^\star \subseteq \left( (FH)^\star : F^\star \right)$. 
	     Conversely, note that $F \left( (FH)^\star : F^\star \right) \subseteq  (FH)^\star$, and so
 we have \ $\left(F \left( (FH)^\star : F^\star \right)\right)^\star \subseteq  (FH)^\star$. Therefore, by the assumption, $\left( (FH)^\star : F^\star \right)^\star \subseteq H^\star$. \end{proof}
 \ec

Using the characterization in Lemma \ref{le:3}, we can associate to any semistar ope\-ration $\,\star\,$ of
$\,D\,$  an \texttt{(e)ab} semistar operation of finite type 
$\, \star_a\, $  of $\,D\,$, called {\it the \texttt{(e)ab}
semistar
operation associated to $\,\star\,$}, \ defined as follows 
for each $ F \in \boldsymbol{f}(D)$ and
for each  $E \in {\overline{\boldsymbol{F}}}(D)$:
$$
\begin {array} {rl}
F^{\star_a} :=& \hskip -5pt  \bigcup\{((FH)^\star:H^\star) \; \ | \; \, \; H \in
\boldsymbol{f}(D)\}\,, \\
E^{\star_a} :=& \hskip -5pt  \bigcup\{F^{\star_a} \; | \; \, F \subseteq E\,,\; F \in
\boldsymbol{f}(D)\}\,, 
\end{array}
$$
\cite[Definition 4.4 and Proposition 4.5]{FL1}. \rm  The previous
construction, in the ideal systems setting, is essentially due to {P. Jaffard}  \cite{J:1960}
and {F. Halter-Koch} \cite{HK:1997}, \cite{HK:1998}.

 Obviously 
 $(\star_{_{\!f}})_{a}= \star_{a}$. Note also that, when $\star = 
\star_{_{\!f}}$, then $\star$ is \texttt{(e)ab} if and only if $\star = 
\star_{a}\,$ \cite[Proposition 4.5(5)]{FL1}. 

It follows that if $\star$ is any \texttt{eab} semistar operation then $\star_{a}$ is the unique  \texttt{(e)ab}  semistar operation which is of finite type and is equivalent to $\star$.  Hence, we can   extend   our previous characterization.

\begin{prro}
Let $D$ be a domain and let $\star$ be an \texttt{eab} semistar operation.  Then $\,\star_{_{\! f}} = {\circlearrowleft_{\mbox{\tiny\rm  Kr}}} = b(\star)  = \star_{a}$.
\end{prro}

\begin{reem}  \rm Note that, with the notation introduced above,  $\,D^{\star_a}\,$ is integrally
closed and contains the integral closure of $\,D\,$ in its field of quotients  \cite[Proposition 4.3 and Proposition 4.5 (10)]{FL1}. In particular,
 when $\,\star = v\,$, then {$\,D^{v_{a}}\,$} 
coincides
 with the pseudo-integral closure of 
$\,D\,$ 
 introduced by {D.F. Anderson}, 
{Houston}
 and {Zafrullah} \cite{AHZ:1991}. Therefore, $\star_a$ is a semistar operation which might be a proper semistar operation, even if $\star$ is a (semi)star operation.
 \end{reem}

The next goal is to show that in many cases the properties \texttt{eab} and \texttt{ab}  coincide. Probably, the most important (semi)star operation which is not   generally of finite type is the $v$--operation.
In this case, from \cite[Theorem 34.6]{G} it follows that  \it the following properties are equivalent:
\begin{enumerate}
\item[(i)] For each $F\in \f(D)$, $(FF^{-1})^v =D$.
\item[(ii)] For each $F\in \f(D)$,  $F$ is $v$--\texttt{ab} (i.e., $v$ is  \texttt{ab}).
\item[(iii)] For each $F\in \f(D)$,  $F$ is $v$--\texttt{eab} (i.e., $v$ is  \texttt{eab} or, equivalently, $D$ is a $v$--domain \cite[page 418]{G}). 

\end{enumerate}
\rm

\medskip

 We have already observed that, for a semistar operation $\star$, if $\star =\starf$, then the notions of $\star$--\texttt{ab} and  $\star$--\texttt{eab} coincide (Proposition \ref{pr:4}). The following result provides further information, but to state it we need to recall some standard facts on semistar operations and related ideals.

We say that a nonzero ideal $I$ of $D$ is a
\emph{quasi-$\star$-ideal} if $I^\star \cap D = I$, a
\emph{quasi-$\star$-prime} if it is a prime quasi-$\star$-ideal,
and a \emph{quasi-$\star$-maximal} if it is maximal in the set of
all  proper  quasi-$\star$-ideals. A quasi-$\star$-maximal ideal is  a
prime ideal. It is possible  to prove that each  proper  quasi-$\star_{_{\!
f}}$-ideal is contained in a quasi-$\star_{_{\! f}}$-maximal
ideal.  More details can be found in \cite[page 4781]{FL:2003}. We
will denote by $\QMax^{\star}(D)$ the set of the 
quasi-$\star$-maximal ideals  of $D$. By the previous considerations we have that 
$\QMax^{\star}(D)$ is not empty, for all semistar operations $\star$ {\sl of finite type}.
Then, for each $E \in \boldsymbol{\overline{F}}(D)$,  we can consider
  $$
  E^{\start} := \bigcap \left \{ED_P \mid P \in  \QMax^{\star_{_{\! f}}}(D) \right\}\,.$$ 
It is well known that the previous definition gives rise to  a semistar operation $\start$ of $D$ which is stable (i.e.,  $(E \cap F)^{\start}
= E^{\start} \cap F^{\start}$, for each $E,F \in \boldsymbol{\overline{F}}(D)$) and of finite type \cite[Corollary 3.9]{FH}. {\rc Recall that, if $K$ is the quotient field of $D$ and $X$ is an indeterminate over $K$,  we set $\Na(D, \star):= \{f/g \in K(X) \mid  f, g \in D[X], 0 \neq g \mbox{ and either } \co(f)^\star \subseteq \co(g)^\star \mbox{ or } f=0 \}$. It is known that $E^{\start} = E \Na(D, \star) \cap K$ for all $E \in \boldsymbol{\overline{F}}(D)$ \cite[Proposition 3.4(3)]{FL:2003}}.

\begin{prro} \label{star-canc} Let $\,\star\,$ be a semistar 
operation of
an integral domain $\,D\,$ and let $F\in \f(D)$.   The following properties are equivalent:
\begin{enumerate}
\item[(i)] $(FF^{-1})^{\starf} =D^\star $.
\item[(ii)]  $F$ is $\start$--\texttt{ab}.
\item[(iii)]  $F$ is $\start$--\texttt{eab}. 

\end{enumerate}

\end{prro}
\begin{proof}  Since $\QMax^{\starf}(D)= \QMax^{\start}(D) $  \cite[Corollary 3.5(2)]{FL:2003},  it is easy to see that $(FF^{-1})^{\starf} =D^\star $ if and only if $(FF^{-1})^{\start} = D^{\start} $. From this observation, it follows immediately that (i)$\Rightarrow$(ii).  By the definitions, it is clear that (ii)$\Rightarrow$(iii).

(iii)$\Rightarrow$(i)  By \cite[Theorem 2.23]{FP}, recall that (i) is also equivalent to each of the following statements:

\begin{enumerate}
\item[(i$^{\prime}$)] $F_Q$ is a nonzero principal fractional ideal in $D_Q$, for all $Q \in \QMax^{\starf}(D)$.
\item[(i$^{{\prime}{\prime}}$)]  $F \Na(D, \star)$ is an invertible fractional ideal of $\Na(D, \star)$ (i.e., $F \Na(D, \star)_M$ is nonzero principal, for each $M \in \Max(\Na(D, \star))$).
\end{enumerate}

 Let $F \in \f(D)$ be a $\widetilde{\star}$--{\texttt{eab}} ideal.
We want to show that $F_Q$ is a nonzero fractional principal ideal of $D_Q$
for all quasi-$\starf$-maximal (= quasi-$\widetilde{\star}$-maximal) ideal
$Q$ of $D$.
Note that, by definition of $\widetilde{\star}$, it is easy to see that
$H^{\widetilde{\star}}D_Q = HD_Q$, for all $H \in \f(D)$ and for all
quasi-$\widetilde{\star}$-maximal ideal $Q$ of $D$.
From this observation and from the fact that each finitely generated ideal
of $D_Q$ is extended from a finitely generated ideal of $D$, it follows that
$F$ $\widetilde{\star}$--{\texttt{eab}}  implies that $FD_Q$ is nonzero
(quasi-)cancellative in $D_Q$, for all quasi-$\widetilde{\star}$-maximal
ideals $Q$ of $D$. This is equivalent to saying that $FD_Q$ is nonzero principal
in $D_Q$, for all quasi-$\widetilde{\star}$-maximal ideal $Q$ of $D$ by
Remark 2(1).   \end{proof}

From the previous proposition, we reobtain some of the characterizations given in \cite[Theorem 3.1]{FJS} of a \emph{Pr\"ufer $\star$--multiplication domain} (i.e., an integral domain in which every nonzero finitely generated ideal is $\starf$--invertible).

\begin{coor} \label{P*MD}  Let $\star$ be a semistar 
operation of
an integral domain $\,D\,$.   The following properties are equivalent:
\begin{enumerate}
\item[(i)] $D$ is a Pr\"ufer $\star$--multiplication domain.
\item[(ii)]   $\start$ is  \texttt{ab}.
\item[(iii)] $\start$ is  \texttt{eab}. 
\item[(iv)] $\Na(D, \star)$ is a Pr\"ufer domain.

\end{enumerate}

\end{coor}
\begin{proof}  (i), (ii) and (iii) are the direct globalizations to all $F \in \f(D)$ of the corresponding properties of Proposition \ref{star-canc}. (iv) is equivalent to the globalization of   (i$^{{\prime}{\prime}}$) to all $F_1 \in \f(\Na(D, \star))$.
\end{proof}

\begin{reem} \bf (1) \rm Note that, even for a star operation (of finite type) $\ast$, the notions of $\astf$--\texttt{ab} and ${\astt}$--\texttt{ab} do not coincide.
For instance, take $\ast$ equal to the $b$--operation on an integrally closed non-Pr\"ufer domain $D$, then clearly $b_{\!_{f}} = b$ and $\widetilde{b} = d$.  Moreover, $b$ is an \texttt{ab}--operation for every domain $D$, but $d$ is not an \texttt{ab}--operation if $D$ is not Pr\"ufer. 
In particular,  the previous example shows  that there exist star operations (of finite type) $\ast$ and nonzero finitely generated ideals that are  $\astt$--\texttt{ab} but not ${\astf}$--\texttt{ab}.  

 \bf (2) \rm From the previous observation and from Corollary \ref{P*MD}, we also deduce that the notions of Pr\"ufer $b$--multiplication domain and Pr\"ufer domain coincide.
 
 \bf (3) \rm  Note that if $\star_1$ and $\star_2$ are two semistar operations on an integral domain and if $\star_1 \leq \star_2$, then in general there are no relations between the notions of $\star_1$--\texttt{eab} (respectively, $\star_1$--\texttt{ab}) ideal and  
 $\star_2$--\texttt{eab} [respectively, $\star_2$--\texttt{ab}] ideal. 
 
 For instance, let $K$ be a field and $X$ and $Y$ two indeterminates over $K$. Set $D:=K[X,Y]_{(X,Y)}$ and $N:=(X,Y)D$. Consider on $D$ the (semi)star operation $\star$ (of finite type) defined in \cite[Example 5.3]{FL:2003}.
  In this case, $\left(\widetilde{\star}\right)_{\!_{a}}= b \leq \star \leq \stara= t$.  (The only fact not already explicitly proved in  \cite[Example 5.3]{FL:2003} is that  $b \leq \star$, but this follows from examining each type of ideal occurring in the set of ideals ``generating''   $\star$  and from observing that $N^b= N$, because there is always a valuation overring of $D$ centered on $N$.) 
  So, obviously, every ideal of $D$ is $b$--\texttt{ab} and $t$--\texttt{ab}, but for instance $N$ is not $\star$--\texttt{ab}, since by definition $(N^k)^\star = N^\star = N$, for all $k \geq 1$. 
  
  On the other hand, in general, we know that, given  $\star_1$ and $\star_2$   two semistar operations on an integral domain with  $\star_1 \leq \star_2$ and  $F\in\f(D)$, then $F$ is $\widetilde{\ \star_1}$--\texttt{ab} implies that $F$ is  $\widetilde{\ \star_2}$--\texttt{ab} 
  \cite[Corollary 5.2(1)]{FL:2007}.
 \end{reem}

\section{Examples}

Now we proceed to the promised examples.

\begin{exxe} \label{notcompleteW}  \sl An example of a $\calbW$--operation which is not of finite type (and so it is not a complete $\calbW$--operation).  \rm 

We say that a domain $D$ is an \emph{almost Dedekind domain} if $D_M$ is a DVR for each maximal ideal $M$ of $D$.  Let $D$ be an almost Dedekind domain with the property that each maximal ideal is finitely generated except for one.  Let $M_\infty$ be the one maximal ideal of $D$ which is not finitely generated.   Explicit examples   of such domains can be found for instance in  \cite[Example 2]{G:1966}, \cite[Example 30]{L:1995}, \cite[6.10]{L:1997} or \cite{L:2006}.   Let $\Delta := \Max(D) \setminus \{M_{\infty}\}$ and let $\calbW :=\{D_M \mid M \in \Delta\}$ and set $\star := \wedge_{\calbW}$. In this case $D^\star = \bigcap \{D_M \mid M \in \Delta\} =  D$   \cite[Example 2]{G:1966}   and so $\star$ is a (semi)star operation of $D$.  

 Since $D$ is a Pr\"{u}fer domain, each nonzero finitely generated ideal $F$ is invertible.    Moreover,  any invertible ideal is necessarily a $v$--ideal and so, in particular, $F^\star = F$ for each $F \in \f(D)$.   If follows that $\starf$ is the identity operation of $D$.  However, it is clear from the definition that $(M_\infty)^\star = D$.  Hence, $\star$ is not of finite type.  
\end{exxe}

For the next example,  we note that \cite[Proposition 32.4]{G} provides a way of producing star operations given a collection of ideals.  In particular, we begin with a collection $\calS$ of fractional ideals of $D$ which contains all of the principal fractional ideals and satisfies the condition that if $J \in \calS$ and $\alpha D$ is a principal fractional ideal of $D$ then $\alpha J \in \calS$.  We then define the star operation $\ast$ of $D$ (depending on $\calS$)  by saying that, for each $E \in \F(D)$,
$$
E^{\ast} := \bigcap \{J  \mid J \in \calS \mbox{ and } E \subseteq J \}\,.
$$

\begin{exxe}  \sl An example of an \texttt{ab} star operation which is not a (star) $\calbW$--operation. \rm

As in the previous example, we let $D$ be an almost Dedekind domain with exactly one maximal ideal $M_\infty$ which is not finitely generated.  To define the required star operation,  we give a generating collection of ideals as in the comments above.  In particular, we let $\calS$ consist of all fractional invertible  ideals and all ideals of the form $JM_\infty$ where $J$ is fractional invertible.   As recalled above,  \cite[Proposition 32.4]{G} guarantees that this collection will generate a star operation $\ast$ of $D$ and it is well known that any star operation on a Pr\"{u}fer domain is an \texttt{ab} operation.    Finally, we note that $M_\infty^2$ cannot be written as an intersection of ideals in $\calS$ (in fact, $M_\infty^2$ is only contained in the ideal $M_\infty$ among the ideals belonging to $\calS$), thus $(M_\infty^2)^\ast = M_\infty$.  This proves that 
$\ast$ cannot be a $\calbW$--operation.   As a matter of fact, clearly,   each valuation overring of $D$ is of the form $D_N$ for some maximal ideal $N$ of $D$.  If $\ast$ were a $\calbW$--operation for some family of valuation overrings $\calbW$ of $D$, then either $D_{M_\infty}$ would be included in $\calbW$ or not.  If it were included, then we would have $(M_\infty^2)^\ast = M_\infty^2$  and, if it were not included, then we would have 
$(M_\infty)^\ast = D$,  as in the previous example.  Both of these possibilities fail in the current example.  It follows that $\ast$ cannot be a $\calbW$--operation.  
\end{exxe}

Finally, we give the example which motivated the paper.

\begin{exxe} \label{ex:1.8} \sl Example of an \texttt{eab} star operation that   is not an \texttt{ab} star operation. \rm

  Let $k$ be a field, let $X_1, X_2, X_n, ...$  be an infinite set of indeterminates over $k$ and let $N:=(X_1, X_2, X_n, ...)k[X_1, X_2, X_n, ...]$. Clearly, $N$ is a maximal ideal in $k[X_1, X_2, X_n, ...]$. Set  $D := k[X_1, X_2, X_n, ...]_N$, let  $M: =ND$  be the maximal ideal of the local domain $D$ and let $K$ be the quotient field of $D$.
  
 Note that $D$ is an UFD and consider $\boldsymbol{\calbW}$ the set of all the rank one valuation overrings of $D$. Let $\wedge_{\calbW}$ be the star \texttt{ab} operation on $D$ defined by $\calbW$. It is well known that the $t$--operation  on $D$ is an \texttt{ab} star operation, since $t\!\!\mid_{\boldsymbol{f}(D)} = 
 \wedge_{\calbW}\!\!\mid_{\boldsymbol{f}(D)}$ \cite[Proposition 44.13]{G}.
  
  We consider the following subset of fractionary ideals of $D$:
  $$  {\calJ}:= \{ xF^{t}, yM, zM^2 \mid x, y, z \in K \setminus \{0\}, \;   F\in \boldsymbol{f}(D)\}\,.
  $$
 Since each nonzero principal fractional ideal  of $D$ is in $ {\calJ}$ and, for each ideal $J \in  {\calJ}$ and for each  nonzero $a \in K$, the ideal $aJ$ belongs to ${\calJ}$, then, as above, \cite[Proposition 32.4]{G}, guarantees that the   set $ {\calJ}$ defines on $D$ a star operation $\ast$.
Since, for each $F \in {\boldsymbol{f}(D)}$,  $F^{t} \in {\calJ}$, then  $\ast\!\!\mid_{\boldsymbol{f}(D)}= t\!\!\mid_{\boldsymbol{f}(D)}$ and so $\ast$ is an \texttt{eab} operation on $D$, since $t$ is an \texttt{(e)ab} 
star operation on $D$. Note that $(X_1, X_2)M \subset M^2$ and  $(M^2)^\ast = M^2$, because  $M^2\in  {\calJ}$.
  
   We claim that:
  $$
  ((X_1, X_2)M)^\ast = ((X_1, X_2))^{t} \cap M^2 =   M^2 = ((X_1, X_2)M^2)^\ast\,.$$  
  
  As a matter of fact,  if $(X_1, X_2)M \subseteq G^{t}$ for some $G \in \boldsymbol{f}(D)$, then  we have \ 
  $((X_1, X_2)D)^{t}M^{t} \subseteq G^{t}$, with $((X_1, X_2)D)^{t} =M^{t} =D$, since $X_1$ and $X_2$ are coprime \ec~in $D$ and so $(X_1, X_2)D$ is not contained in any proper principal ideal of $D$.  Therefore $(X_1, X_2)M$ is not contained in any nontrivial ideal of the type
  $xF^{t} \ (=G^{t} )\in {\calJ}$.
  
   A similar argument shows that $(X_1, X_2)M$ is neither contained in any ideal of the type
  $yM, zM^2\in  {\calJ}$, with $y$ and $z$ nonzero and non unit in $D$,  and thus  the only nontrivial ideals of $ {\calJ}$ containing $(X_1, X_2)M$ are $M^2$ and $M$, hence $
  ((X_1, X_2)M)^\ast = M^2$.  A similar argument shows  that $((X_1, X_2)M^2)^\ast= M^2 $.
  
Since  $ ((X_1, X_2)M)^\ast= M^2= ((X_1, X_2)M^2)^\ast$,  if \  $\ast$ \ were an \texttt{ab} star operation, then we would deduce that $M^\star  \  (=M)$ is equal to $(M^2)^\ast \ (=M^2)$, which is not the case.
  \end{exxe}


\section{Generalization:  a conjecture }

Given an  {\texttt{eab}}  semistar operation $\star$ on a domain $D$, we introduced in Section 2 several natural means to associate a new {\texttt{eab}}  semistar operation to $\star$.  In \cite{FL:2007} we introduced a ring construction   KN$(D, \star)$ which simultaneously generalizes the notions of Kronecker function ring and Nagata ring, for  an arbitrarily given semistar operation $\star$ on any domain $D$.
 Along with this generalized function ring, we introduced a semistar operation $ {\star_{_{\! \ell}} }$ which is a semistar operation on $D$.      What is noteworthy about this is that $ {\star_{_{\! \ell}} }$ possesses at least two different interpretations that seem to be natural generalizations of the constructions giving rise to  the semistar operations ${\circlearrowleft_{\mbox{\tiny\rm  Kr}}}$ and  $b(\star)$ (both coinciding with $\starf$, when $\star$ is {\texttt{eab}}).    On the other hand, $\,\star_{_{\! f}}$ and $ {\star_{_{\! \ell}} }$ can be dramatically different for a given semistar operation $\star$ which is not  {\texttt{eab}}.     What seems plausible then is that we can unify the theory of 
$ {\star_{_{\! \ell}} }$ and the theory developed in this paper regarding  {\texttt{eab}} operations of finite type if we can give a construction for $ {\star_{_{\! \ell}} }$ which agrees with $\,\star_{_{\! f}}$ in the case where $\star$ is  {\texttt{eab}}. We have a candidate for such a construction which seems plausible, but at this time do not have a proof.  We start by giving a brief summary of results from \cite{FL:2007}.

\medskip

Let $D$ be a domain  with quotient field $K$  and let $\star$ be a semistar operation on $D$.  We call an overring $T$ of $D$ a \emph{$\star$--monolocality of } $D$ provided 
 $T^{\star_{_{\! f}}} = T$   and $FT$ is a principal fractional ideal of $T$, for each $\star$--{\texttt{eab}} $F \in \f(D)$.    Let $\boldsymbol{\mathcal{L}}(D,\star)$ be the set of   all 
$\star$--monolocalities on $D$.    We can then define the new semistar operation 
$ {\star_{_{\! \ell}} }$ on $D$   by setting, for each  $E \in {\overline{\boldsymbol{F}}}(D)$, 
$$
  E^{\star_\ell} := \bigcap \{ ET  \mid T \in \boldsymbol{\mathcal{L}}(D,\star) \}\,.
 $$
 {\rc In particular, since a finitely generated ideal extends to a principal ideal in a valuation overring,   we have  $\calbV(\star) \ (=\! \{ V \mid V  \mbox{ is a $\star$--valuation  overring of $D$}\}) \subseteq \boldsymbol{\mathcal{L}}(D,\star) $}. 
 Therefore,  $\star_\ell  \leq b(\star) $.

 \begin{reem} \rm \bf (1) \rm  Note that, for any semistar operation $\star$,  it is known that ${\star_{_{\! \ell}} } \leq \starf$  \cite[Proposition 6.3]{FL:2007}  and this inequality  is  stronger than ${\star_{_{\! \ell}} } \leq b(\star)$, since by definition of $\star$--valuation overring it follows immediately that $\starf \leq b(\star)$.
 
 \bf (2)  \rm Since each $\star$--monolocality contains a minimal $\star$--monolocality  \cite[Proposition 5.11(7)]{FL:2007}, if we denote by $ \boldsymbol{\mathcal L}(D, \star)_{\mbox{\tiny min}}$, or simply by   ${\boldsymbol{\mathcal L}}_{\mbox{\tiny min}}$, the set of all minimal $\star$--monolocalities of $D$, then  $
  E^{\star_\ell} = \bigcap \{ ET  \mid T \in{\boldsymbol{\mathcal L}}_{\mbox{\tiny min}}  \}
 $,  for each  $E \in {\overline{\boldsymbol{F}}}(D)$. 

 \end{reem}
 
If we  define the domain KN$(D,\star)$ to be the subring of the field of rational functions $K(X)$ given by 
$$
\mbox{KN}(D, \star):= 
 \left\{ \frac{f}{g} \mid f, g\in D[X],\ g \neq 0,\ \co(f)^\star \subseteq \co(g)^\star, \mbox{ and} \  \co(g) \mbox{ is }  \star\mbox{--{\texttt{eab}}} \right\}\,,  $$
then we know that this ring generalizes  both the classical Kronecker function ring construction (the case where $\star = \star_a$, i.e., $\mbox{KN}(D, \star_a) = \mbox{Kr}(D, \star_a)= \mbox{Kr}(D, \star)$  \cite[Proposition 5.4(2) and Theorem 5.11(7)]{FL:2007})  and the Nagata ring construction (the case where 
$\star = \widetilde{\star}$, i.e.,  $\mbox{KN}(D, \start) = \mbox{Na}(D, \start)=  \mbox{Na}(D, \star)$  \cite[Proposition 5.4(1) and Theorem 5.1(7)]{FL:2007}). 

As we did for ${\circlearrowleft_{\mbox{\tiny\rm  Kr}}}$, we can then define a new semistar operation ${\circlearrowleft_{\mbox{\tiny\rm  KN}}}$ on $D$ using the previous construction as follows, for each  $E \in {\overline{\boldsymbol{F}}}(D)$, 
$$
 E^{\circlearrowleft_{\mbox{\tiny\rm  KN}}} := E(\mbox{KN}(D,\star)) \cap K\,.
 $$

  Since $b(\star) ={\circlearrowleft_{\mbox{\tiny\rm  Kr}}}$ and  ${\star_{_{\! \ell}} }$  is a generalization   of $b(\star)$,  the key point of this speculative section is then made clear by the following result. 
  
  \begin{thee}{ \rm  \cite[Propositions 5.1, 5.11(7), and  6.3]{FL:2007}}  Let $D$ be a domain and let $\star$ be a semistar operation on $D$.  Then,
  \begin{enumerate}
  \item[(1)] $ \mbox{\rm KN}(D, \star) =  \bigcap \{ T(X)  \mid T \in \mathcal{L}(D,\star) \}$.

\item[(2)] $ {\star_{_{\! \ell}} } = {\circlearrowleft_{\mbox{\tiny\rm  KN}}}$.

\end{enumerate}
\end{thee}

%
%


%

%


%

%

%




%
\ec
\medskip

As noted earlier, what is needed now to unify the theory is a construction which begins with a semistar operation $\star$ and yields ${\star_{_{\! \ell}} }$ in the general setting, but obviously yields $ \star_{_{\! f}} = b(\star)$ in the special case where $\star$ is an {\texttt{eab}} operation  (recall that, if $\star$ is an {\texttt{eab}} semistar operation, then any finitely generated ideal is an {\texttt{eab}} ideal  and so $\mbox{KN}(D, \star) = \mbox{Kr}(D, \star)$), i.e., ${\star_{_{\! \ell}} } =\stara \ (= b(\star) =\starf$)). 

 Let $\f_{\!\flat}(D)$ be the set of all nonzero (finitely generated) $\star$--{\texttt{eab}}  fractional ideals of $D$.
  For each  $E \in {\overline{\boldsymbol{F}}}(D)$, we then define
  $$E^{\star_{{\flat}}} := \bigcup \{ F^\star  \mid F \subseteq E \mbox{ and } F \in\f_\flat(D)
   \}\,.$$
   
   Obviously, $E^{\star_{{\flat}}} \subseteq E^{\starf}$  and   if $\star$ is an {\texttt{eab}} semistar operation, then 
  $\f_{\!\flat}(D) = \f(D)$ and so $\star_{_{\flat}} = \starf$  Note also that, $d_{_{\flat}} = d$.   
  
  \medskip
 
   In general, it is not clear that $E^{\star_{{\flat}}}$ is even an ideal.  So the proposed new semistar operation  would be defined using the ideal generated by the set $E^{\star_{{\flat}}}$.  The reason that such a definition seems reasonable can be seen if one considers how an ideal of $D$ gets larger when one extends to the ring 
  $ \mbox{\rm KN}(D, \star) $ and then contracts to K.  Suppose then that $J$ is an ideal of a domain $D$ and that $\star$ is a semistar operation on $D$.  Suppose also that $I$ is a $\star$--{\texttt{eab}} ideal of $D$ such that $I \subseteq J$.  Let 
 $\{a_0, a_1, \ldots , a_n\}$ be a set of generators of $I$ and let $d \in I^\star$.  Let 
 $f(X) := a_nX^n + \ldots + a_1X + a_0$.  Then by definition $\frac{d}{f(X)} \in  \mbox{\rm KN}(D, \star)$.  It follows that 
 $d \in J^{\star_{_{\! \ell}} }$.  It remains to be proven that $J^{\star_{_{\! \ell}} }$ can be generated by such elements.

   We conclude with the following. 
   \medskip
   
   \noindent \bf Conjecture. \sl Let $D$ be a domain and let $\star$ be a semistar operation on $D$.  Then, $\star_{_{\flat}}$,
as defined above, is a semistar operation on $D$ and
 is actually equal to ${\star_{_{\! \ell}} }$. \rm  

\ec



\begin{thebibliography}{[XX] }
     
      \medskip


\bibitem[AA-1984]{AA}  D. D. Anderson and D. F. Anderson, \it Some remarks on
cancellation ideals, \rm \ Math.  Japonica \bf 6 \rm (1984), 879--886.

\bibitem[AR-1997]{AR:1997}  D. D. Anderson and M. Roitman, \it A characterization of cancellation ideals, \rm Proc. Am. Math. Soc. \bf 125 \rm (1997), 2853--2854.

  
   \bibitem[AHZ-1991]{AHZ:1991}
    {D. F. Anderson, E. G.  Houston and M.  Zafrullah},
   {\em Pseudo-integrality},
  Canad.  Math. Bull. \textbf{34} (1991), 15--22.


  \bibitem[FJS-2003]{FJS}
  M. Fontana, P. Jara, and E. Santos, \emph{Pr\"ufer $\star$-multiplication domains and semistar operations}, J. Algebra Appl. \bf 2 \rm (2003), 21--50.
  
  \bibitem[FH-2000]{FH}
M. Fontana and J.A. Huckaba, \emph{Localizing systems and
semistar
  operations},  in ``Non-Noetherian Commutative Ring Theory'' (edited by S.~T. Chapman and
  S. Glaz ), Kluwer Academic Publishers, 2000,  169--198.
  


  \bibitem[FL-2001a]{FL1}  M. Fontana and K. A. Loper, \it Kronecker function rings:
  a general approach,
   \ \rm \rm in  ``Ideal Theoretic Methods in Commutative Algebra" (D.D.
   Anderson and I.J. Papick, Eds.), M. Dekker Lecture Not
   es Pure Appl.
  Math. \bf 220 \rm (2001), 189--205.
  
  
  \bibitem[FL-2001b]{FL:2001b}  M. Fontana and K. A. Loper, \it A Krull-type theorem for the semistar integral closure of an integral domain, \rm Commutative algebra. AJSE, Arab. J. Sci. Eng. Sect. C Theme Issues \bf 26 \rm (2001), 89--95.


  \bibitem[FL-2003]{FL:2003} M. Fontana and K. A. Loper, \it  Nagata rings, Kronecker function rings and related semistar operations, \rm  Comm. Algebra \bf 31 \rm (2003), 4775--4805. 

 
   \bibitem[FL-2007]{FL:2007}
M. Fontana and A. Loper, \emph{A generalization of Kronecker function rings and Nagata rings}, Forum Math. \bf 19 \rm (2007), 971--1004. 

 
  
  \bibitem[FP-2005]
  {FP}  M. Fontana and G. Picozza, \it Semistar invertibility on integral domains, \rm  Algebra Colloquium \bf 12 \rm (2005), 645--664.  


  \bibitem[G-1965]{G65}  R. Gilmer, \it The cancellation law for ideals in commutative rings, \rm Canad. J. Math.  \bf 17 \rm (1965), 281--287.
  
   
  \bibitem[G-1966]{G:1966}  R. Gilmer, \it Overrings of Pr\"ufer domains, \rm  J. Algebra  \bf 4 \rm  (1966), 331--340.
  
  
  \bibitem[G-1968]{G68}  R. Gilmer, Multiplicative Ideal Theory, Part I and Part II, Queen's Papers on Pure Appl. Math. \bf 12 \rm, 1968.
  
  \bibitem[G-1972]{G}  R. Gilmer, Multiplicative Ideal Theory, M.
  Dekker, New York, 1972.



 \bibitem[GO-2000]{GO:2000} H. P. Goeters and B. Olberding, \it On the multiplicative properties of submodules of the quotient field of an integral domain, \rm \bf 26 \rm (2000), 241--254.
 
 
 \bibitem[HK-1997]{HK:1997} F.  Halter-Koch, \it Generalized integral closures, \ \rm
 in ``Factorization in Integral Domains" (D.D.  Anderson, Ed.).  M.
  Dekker Lecture Notes Pure Appl.
 Math. \bf 187 \rm (1997),
 349--358.
%
 
 \bibitem[HK-1998]{HK:1998} F. Halter-Koch, Ideal Systems: An Introduction to
   Multiplicative Ideal Theory,  M.  Dekker, New York, 1998.


  \bibitem[J-1960]{J:1960}  P. Jaffard, Les Syst\`emes d'Id\'eaux, Dunod, Paris, 1960.

\bibitem[Je-1963]{Jensen:1963} C. H. Jensen, \it On characterizations of Pr\"ufer domains, \rm  Math. Scand. \bf 13 \rm (1963), 90--98.

\bibitem[Ka-1971]{Ka:1971} I. Kaplansky, Topics in commutative ring theory. Unpublished notes. Chicago, 1971.
  
    \bibitem[K-1936]{Krull:1936}   W. Krull, \it Beitr\"age zur Arithmetik kommutativer
    Integrit\"atsbereiche, \rm I - II.  \ Math.  Z.  {\bf 41} (1936),
   545--577; $\,$ 665--679.
   
\bibitem[K-1999]{Krull}W. Krull,
Gesammelte Abhandlungen / Collected Papers,
Hrsg. v. Paulo Ribenboim,  Walter de Gruyter, Berlin,
1999. 
   
   
   \bibitem[L-1995]{L:1995}
   A. Loper, \emph{More almost Dedekind domains and Pr\"ufer domains of polynomials}, in   ``Zero-dimensional commutative rings'' (Knoxville, TN, 1994),  287--298, Lecture Notes in Pure and Appl. Math., {\bf 171}, Dekker, New York, 1995.
   
      \bibitem[L-1997]{L:1997}
   A. Loper, \emph{Sequence domains and integer-valued polynomials},
J. Pure Appl. Algebra \bf 119 \rm  (1997),  185--210.
   
   \bibitem[L-2006]{L:2006}
   A. Loper, \emph{Almost Dedekind domains that are not Dedekind}, in ``Multiplicative ideal theory in commutative algebra. A tribute to the work of Robert Gilmer'' (edited by W. Brewer, S. Glaz, W. Heinzer, and B. Olberding ) 276--292, Springer, New York, 2006.
 
   
    \bibitem[OM-1994]{OM2}A. Okabe and R. Matsuda, {\it Semistar operations on
   integral domains},\ Math.
   J. Toyama Univ. {\bf 17} (1994), 1--21.
     
     
   \bibitem[WMc-1997]{WMc1} Wang Fanggui and R.L. McCasland, \emph{On $w$-modules over strong Mori domains}, Comm. Algebra \bf 25 \rm (1997), 1285--1306.
   
      \bibitem[WMc-1999]{WMc2} Wang Fanggui and R.L. McCasland, \emph{On  strong Mori domains}, J. Pure Appl. Algebra \bf 135 \rm  (1999), 155--165.

   \bibitem[ZS-1960]{ZS} O. Zariski and P. Samuel, Commutative Algebra, Vol. II, Van Nostrand, New York, 1960. 

     
  \end{thebibliography}
\end{document}